\documentclass[a4paper,10pt]{article}

\usepackage{geometry}
\usepackage{amsmath}
\usepackage{amssymb}
\usepackage{amsthm}
\usepackage{mathrsfs}
\usepackage{color}
\usepackage{amscd}
\usepackage{graphicx}
\usepackage{hyperref}

 \newtheorem{thm}{Theorem}[section]
 \newtheorem{cor}[thm]{Corollary}
 
 \newtheorem{prop}[thm]{Proposition}
 \theoremstyle{definition}
 \newtheorem{defn}[thm]{Definition}
 \theoremstyle{remark}
 \newtheorem{rem}[thm]{Remark}
 
 \newtheorem*{examples}{Examples}
 \numberwithin{equation}{section}

\newcommand{\C}{\mbb{C}}
\newcommand{\HH}{\mbb{H}}
\newcommand{\Oc}{\mbb{O}}

\newcommand{\I}{\mc{I}}

\newcommand{\N}{\mbb{N}}

\newcommand{\R}{\mbb{R}}

\def\Q{\mc{Q}}

\newcommand{\OO}{\Omega}

\newcommand{\mbb}{\mathbb}
\newcommand{\mc}{\mathcal}

\newcommand{\ui}{\imath}

\newcommand{\dibar}{\overline\partial}
\newcommand{\dcrf}{\dibar_{\scriptscriptstyle CRF}}
\newcommand\sd[1]{{#1}'_s}
\def\SS{\mathbb S}
\newcommand\Ac{A\otimes_{\R}{\C}}

\newcommand\dd[2]{\dfrac{\partial#1}{\partial#2}} 
\newcommand\ddd[2]{\frac{\partial#1}{\partial#2}} 
\newcommand\dds[1]{\partial_{#1}}

\newcommand\IM{\operatorname{Im}}
\newcommand\RE{\operatorname{Re}}

\newcommand{\dif}{\vartheta}
\newcommand{\difbar}{\overline{\dif}}

\newcommand{\bc}{\begin{center}}
\newcommand{\ec}{\end{center}}
\newcommand{\Rn}{\mathcal{\R}_{n}}
\newcommand\vs[1]{{#1}_s^\circ}

\title{\bf Slice regularity and harmonicity\\ on Clifford algebras}

\author{A.~Perotti\footnote{Partially supported by GNSAGA of INdAM}
\\
\small Dipartimento di Matematica, Universit\`a di Trento\\ 
\small Via Sommarive 14, I-38123 Povo Trento, Italy\\
\small alessandro.perotti@unitn.it}

\date{  }

\begin{document}

\maketitle

\begin{abstract}
We present some new relations between the Cauchy-Riemann operator on the real Clifford algebra $\Rn$ of signature $(0,n)$ and  slice-regular functions on $\Rn$. The class of slice-regular functions, which comprises all polynomials with coefficients on one side, is the base of a recent function theory in several hypercomplex settings, including quaternions and Clifford algebras. In this paper we present formulas, relating the Cauchy-Riemann operator, the spherical Dirac operator, the differential operator characterizing slice regularity, and the {spherical derivative} of a slice function. The computation of the Laplacian of the spherical derivative of a slice regular function gives a result which implies, in particular, the Fueter-Sce Theorem. In the two four-dimensional cases represented by the paravectors of $\mathbb R_3$ and by the space of quaternions, these results are related to zonal harmonics on the three-dimensional sphere and to the Poisson kernel of the unit ball of $\R^4$.
\\


\end{abstract}

\section{Introduction and preliminaries}
Let $\Rn$ denote  the real Clifford algebra of signature $(0,n)$, with basis vectors $e_1,\ldots, e_n$. Consider the Dirac operator
\[\mathcal D=e_1\dd{}{x_1}+\cdots+e_n\dd{}{x_n}\]
and the Cauchy-Riemann operators 
\[\dibar=\dd{}{x_0}+e_1\dd{}{x_1}+\cdots+e_n\dd{}{x_n}\text{\quad and\quad}\partial=\dd{}{x_0}-e_1\dd{}{x_1}-\cdots-e_n\dd{}{x_n}\]
on $\Rn$. The operator $\dibar$ factorizes the Laplacian operator 
\[\partial \dibar=\dibar\partial=\Delta_{n+1}=\frac{\partial^2}{\partial x_0^2}+\frac{\partial^2}{\partial x_1^2}+\cdots+\frac{\partial^2}{\partial x_n^2}\]
of the paravector space \[V=\{x_0+x_1e_1+\cdots +x_ne_n\in\Rn\ |\ x_0,\ldots,x_n\in\R\}\simeq\mathbb R^{n+1}.\] 
This property is one of the most attractive aspects of Clifford analysis, the well-developed function theory based on Dirac and Cauchy-Riemann operators (see \cite{BDS,DSS,GHS} and the vast bibliography therein).

In this paper we prove some new relations between the Cauchy-Riemann operator, the spherical Dirac operator, the Laplacian operator and the class of \emph{slice-regular} functions on a Clifford algebra. Slice-regular functions constitute a recent but rapidly expanding function theory in several hypercomplex settings, including quaternions and real Clifford algebras (see \cite{GS, CSS,GhPe_AIM, GSS,AlgebraSliceFunctions,DivisionAlgebras}). 
This class of functions was introduced by Gentili and Struppa \cite{GS} in 2006-2007  for functions of one quaternionic variable. Let $\HH$ denote the skew field of quaternions, with basic elements $i,j,k$. For each quaternion $J$ in the sphere of imaginary units
 \[\SS_\HH=\{J\in\HH\ |\ J^2=-1\}=\{x_1i+x_2j+x_3k\in\HH\ |\ x_1^2+x_2^2+x_3^2=1\},\]
let $\C_J=\langle 1,J\rangle\simeq\C$ be the subalgebra generated by $J$. Then we have the ``slice'' decomposition
\[\HH=\bigcup_{J\in \SS_\HH}\C_J, \quad\text{with $\C_J\cap\C_K=\R$\quad for every $J,K\in\SS_\HH,\ J\ne\pm K$.}\]

A differentiable function $f:\OO\subseteq\HH\rightarrow\HH$  is called \emph{(left) slice-regular} on $\OO$ if, for each $J\in\SS_\HH$, the restriction
\[f_{\,|\OO\cap\C_J}\ : \ \OO\cap\C_J\rightarrow \HH\]
is holomorphic with respect to the complex structure defined by left multiplication by $J$. For example, polynomials $f(x)=\sum_m x^ma_m$ with quaternionic coefficients on the right are slice-regular on $\HH$. More generally, convergent power series are slice-regular on an open ball centered at the origin.
Observe that nonconstant polynomials 
do not belong to the kernel of the Cauchy-Riemann-Fueter operator \[\dcrf =\dd{}{x_0}+i\dd{}{x_1}+j\dd{}{x_2}+k\dd{}{x_3}.\]
Here $x_0,x_1,x_2,x_3$ denote the real components of a quaternion $x=x_0+x_1i+x_2j+x_3k$.

Let $\HH\otimes_{\R}\C$ be the algebra of complex quaternions, with elements $w=a+\ui b$, $a,b\in\HH$, $\ui^2=-1$.
Every quaternionic polynomial $f(x)=\sum_m x^ma_m$ lifts to a unique polynomial function $F:\C\rightarrow\HH\otimes_{\R}\C$  which makes the following diagram commutative for all $J \in \SS_\HH$:
\[
\begin{CD}
\C\simeq \R\otimes_\R\C @>F> >\HH\otimes_\R\C\\ % right arrow with labels
@V \Phi_J V  V %down arrow with labels
@V V \Phi_J V\\%down arrow with labels
\HH @>f> >\HH % right arrow with labels
\end{CD} 
\]
where $\Phi_J: \HH\otimes_{\R}\C \to \HH$ is defined by $\Phi_J(a+\ui b):=a+Jb$. The lifted polynomial is simply $F(z)=\sum_m z^ma_m$, with variable $z=\alpha+\ui \beta\in\C$.
This lifting property is equivalent to the following fact:
for each $z=\alpha+\ui\beta\in\C$, the restriction of $f$ to the 2-sphere $\alpha+\SS_\HH\beta=\cup_{J\in\SS_\HH}\Phi_J(z)$, is a quaternionic left-affine function with respect to\ $J\in\SS_\HH$, namely of the form $J\mapsto a+Jb$ ($a,b\in\HH$). 

In this lifting, the usual product of polynomials with coefficients in $\HH$ on one fixed side (the one obtained by imposing commutativity of the indeterminate with the coefficients when two polynomials are multiplied together) corresponds to the pointwise pro\-duct in the algebra $\HH\otimes_{\R}\C$.

More generally, if a quaternionic function $f$ (not necessarily a polynomial) has a holomorphic lifting $F$ then $f$ is called \emph{(left) slice-regular}.

This approach to slice regularity can be pursued on an ample class of real algebras. Here we recall the basic definitions and refer to \cite{GhPe_AIM,AlgebraSliceFunctions} for details and other references. Let $A$ be a real alternative algebra with unity $e$. The real multiples of $e$ in $A$ are identified with the real numbers. Assume that $A$ is a *-algebra, i.e.\ it is equipped with a linear antiinvolution $x\mapsto x^c$, such that $(xy)^c=y^cx^c$ for all $x,y\in A$ and  $x^c=x$ for $x$ real. Let $t(x):=x+x^c\in A$ be the \emph{trace} of $x$ and $n(x):=xx^c\in A$  the \emph{norm} of $x$. 
Let \[\SS_A:=\{J\in A\ |\ t(x)=0,\ n(x)=1\}\]
be the ``sphere'' of the imaginary units of $A$ compatible with the *-algebra structure of $A$. Assuming $\SS_A\ne\emptyset$, one can consider the \emph{quadratic cone} of $A$, defined as the subset of $A$
  \[\Q_A:=\bigcup_{J\in \SS_A}\C_J\]
where $\C_J=\langle 1,J\rangle$ is the complex ``slice'' of $A$ generated as a subalgebra by $J\in\SS_A$. It holds $\C_J\cap\C_K=\R$ for each $J,K\in\SS_A$, $J\ne\pm K$. The quadratic cone is a real cone invariant with respect to translations along the real axis.
	
Observe that $t$ and $n$ are real-valued on $\Q_A$ and that $\Q_A=A$ if and only if $A\simeq\C,\HH,\Oc$ (where $\Oc$ is the algebra of octonions).

The remark made above about quaternionic polynomials suggests a way to define polynomials with coefficients in $A$ or more generally $A$-valued functions on the quadratic cone of $A$.
Let $J\in\SS_A$ and let $\Phi_J: A\otimes_{\R}\C \to A$ defined by $\Phi_J(a+\ui b):=a+Jb$ for any $a,b\in A$. 
By imposing commutativity of diagrams
\begin{equation}\label{diagram}
\begin{CD}
\C\simeq \R\otimes_{\R}\C @>F> >A\otimes_{\R}\C\\ % right arrow with labels
@V \Phi_J V  V %down arrow with labels
@V V \Phi_J V\\%down arrow with labels
\Q_A @>f> >A % right arrow with labels
\end{CD} 
\end{equation}
for all $J \in \SS_A$, we get the class of \emph{slice functions} on $A$. This is the class of functions which are compatible with the slice character of the quadratic cone.

More precisely, let $D\subseteq\C$ be a set that is invariant with respect to complex conjugation. 
In $\Ac$ consider the complex conjugation mapping $w=a+\ui b$ to $\overline w=a-\ui b$ ($a,b\in A$).
If a function $F: D \to A\otimes_{\R}\C$ satisfies  $F(\overline z)=\overline{F(z)}$ for every $z\in D$, then $F$  is called a \emph{stem function} on $D$. Let $\OO_D$ be the \emph{circular} subset of the quadratic cone defined by 
\[\OO_D=\bigcup_{J\in\SS_A}\Phi_J(D).\]
 The stem function $F=F_1+\ui F_2:D \to A\otimes_{\R}\C$  induces the \emph{(left) slice function} $f=\I(F):\OO_D \to A$ in the following way: if $x=\alpha+J\beta =\Phi_J(z)\in \OO_D\cap \C_J$, then  
\[ f(x)=F_1(z)+JF_2(z),\]
where $z=\alpha+\ui\beta$.
The slice function $f=\I(F)$ is called \emph{(left) slice-regular} if $F$ is holomorphic. 
The function $f=\I(F)$ is called \emph{slice-preserving} if $F_1$ and $F_2$ are real-valued (this is the case already considered by Fueter \cite{F} for quaternionic functions and by  G\"urlebeck and Spr\"ossig (cf.\ \cite{GuS,GHS}) for \emph{radially holomorphic} functions on Clifford algebras). In this case, the condition $f(x^c)=f(x)^c$ holds for each $x\in\OO_D$.

When $A$ is the algebra of real quaternions and the domain $D$ intersects the real axis, this definition of slice regularity is equivalent to the one proposed by Gentili and Struppa \cite{GS}.

In this paper we are mainly interested in the case where $A$ is the real $2^n$-dimensional Clifford algebra $\Rn$ with signature $(0,n)$. Let $e_1,\ldots, e_n$ be the basic units of $\Rn$, satisfying $e_ie_j+e_je_i=-2\delta_{ij}$ for each $i,j$, and let $e_K$ denote the basis elements $e_K=e_{i_1}\cdots e_{i_k}$, with $e_\emptyset=1$ and $K=(i_1,\ldots,i_k)$ an increasing multi-index of length $k$, with $0\le k \le n$. Every element $x\in\Rn$ can be written as $x=\sum_{K}x_Ke_K$,  with $x_K\in\R$.  We will identify paravectors, i.e.\ elements $x\in\Rn$ of the form $x=x_0+x_1e_1+\cdots +x_ne_n$, with elements of the Euclidean space $\R^{n+1}$.

The Clifford conjugation $x\mapsto x^c$  is the unique antiinvolution of $\Rn$ such that $e_i^c=-e_i$ for $i=1,\ldots,n$. 
If $x=x_0+x_1e_1+\cdots +x_ne_n\in\R^{n+1}$ is a paravector, then $x^c=x_0-x_1e_1-\cdots -x_ne_n$. Therefore $t(x)=x+x^c=2x_0$ and $n(x)=xx^c=|x|^2$, the squared Euclidean norm. The same formulas for $t$ and $n$ hold on the entire quadratic cone of $\Rn$, that in this case  (cf.\ \cite{GhPe_AIM,GhPe_Trends}) can be defined as
\[\Q_{\Rn}=\{x\in\Rn\;|\; t(x)\in\R,\,n(x)\in\R\}.\]
The quadratic cone of $\Rn$ is a real algebraic set which contains the paravector space $\R^{n+1}$ as a proper (if $n>2$) subset.
For example, $\Q_{\R_{1}}=\R_{1}\simeq\C$, $\Q_{\R_{2}}=\R_{2}\simeq\HH$, while
\[\Q_{\R_3}=\{x\in\R_3\;|\;x_{123}=x_1x_{23}-x_2x_{13}+x_3x_{12}=0\}\] is a real algebraic set of dimension 6. 

Each $x\in\Q_{\Rn}$ can be written as $x=\RE(x)+\IM(x)$, with $\RE(x)=\frac{x+x^c}2$, $\IM(x)=\frac{x-x^c}2=\beta J$, where $\beta=|\IM(x)|$ and $J\in\SS_{\Rn}$ (the ``sphere'' of imaginary units in $\Rn$ compatible with the Clifford conjugation). The choice of $J$ is unique if $x\not\in\R$.

The class of slice-regular functions on $\Rn$, defined as explained before by means of holomorphic stem functions, is an extension of the class of \emph{slice-monogenic} functions introduced by Colombo, Sabadini and Struppa in 2009 \cite{CSS}. More precisely, let $\SS^{n-1}=\{x_1e_1+\cdots+x_ne_n\in\Rn\,|\,x_1^2+\cdots+x_n^2=1\}$, a subset of $\SS_{\Rn}$. 
A function $f:\OO\subseteq\R^{n+1}\to\Rn$ is slice-monogenic if, for every $J\in\SS^{n-1}$, the restriction $f_{\,|\OO\cap\C_J}\, : \, \OO\cap\C_J\rightarrow \Rn$ is holomorphic with respect to the complex structure  $L_J$ defined by $L_J(v)=Jv$.
When the domain $\OO$ intersects the real axis, the definition of slice regularity on $\Rn$ is equivalent to the one of slice monogenicity, in the sense that the restriction to the paravector space of a $\Rn$-valued slice-regular function is a slice-monogenic function.

In the next sections we will introduce a differential operator $\difbar$ characterizing slice regularity \cite{CGS,Gh_Pe_GlobDiff} and the notion of \emph{spherical derivative} of a slice function \cite{GhPe_AIM}. Then we will prove a formula, relating the  operator $\difbar$, the spherical derivative, the spherical Dirac operator and the Cauchy-Riemann operator on $\Rn$.
The computation of the Laplacian of the spherical derivative of a slice regular function gives a result which implies, in particular, the Fueter-Sce Theorem for  monogenic functions (i.e.\ the functions belonging to the kernel of the Cauchy-Riemann operator $\dibar$). We recall that Fueter's Theorem \cite{F}, generalized by Sce \cite{Sce}, Qian \cite{Qian1997} and Sommen \cite {Sommen2000} on Clifford algebras and octonions, in our language states that applying  to a slice-preserving slice-regular function the Laplacian operator of $\R^4$ (in the quaternionic case) or the iterated Laplacian operator $\Delta_{n+1}^{{(n-1)}/2}$ of $\R^{n+1}$ (in the Clifford algebra case with $n$ odd), one obtains a function in the kernel, respectively, of the Cauchy-Riemann-Fueter operator or of the Cauchy-Riemann operator.
This result was extended  in \cite{CSSFueter2010} to the whole classes of quaternionic  slice-regular functions and of slice-monogenic functions defined by means of stem functions.

These results take a particularly neat form in the two four-dimensional cases represented by paravectors in the Clifford algebra $\mathbb R_3$ and by the space $\HH$ of quaternions. As we will show in Sections~\ref{4dcase} and \ref{Thequaternioniccase}, in these cases there appear unexpected relations  between slice-regular functions, the zonal harmonics on the three-dimensional sphere and the Poisson kernel of the unit ball.

As an application, a stronger version of Liouville's Theorem for entire slice-regular functions is given. See \cite{GS,GSS,CSS} for the generalization of the complex Liouville's Theorem to quaternionic functions and slice monogenic functions. 
The formulas obtained in the present paper have found application also to four dimensional Jensen formulas  for quaternionic slice-regular functions \cite{AltavillaBisi,FourJensen}.

The present work can be considered also a continuation of \cite{indam2013}, where other relations, of a different nature, between the two function theories, the one of monogenic functions and the one of slice-regular functions, were presented.

%%-- Numbered sections
\section{The slice derivatives and the operator $\difbar$}

The commutative diagrams \eqref{diagram} suggest a natural definition of the \emph{slice derivatives} $\ddd{f}{x},\ddd{f\;}{x^c}$ of a slice functions $f$. They are the slice functions induced, respectively, by the derivatives $\ddd{F}{z}$ and $\ddd{F}{\overline z}$:
\[\dd{f}{x}=\I\left(\dd{F}{z}\right)\quad\text{and}\quad \dd{f\;}{x^c}=\I\left(\dd{F}{\overline z}\right).\]
With this notation a slice function is slice-regular if and only if $\ddd{f\;}{x^c}=0$ and if this is the case also the slice derivative $\ddd{f}{x}$ is slice-regular.

For each alternative $^*$-algebra $A$ there exists  \cite{Gh_Pe_GlobDiff} a global differential operator $\difbar$ which characterizes slice-regular functions among the class of slice functions. If $\OO_D$ is a circular domain in the quadratic cone of $A$, the operator 
\[\difbar:\mathcal{C}^1(\OO_D \setminus \R,A) \to \mathcal{C}^0(\OO_D\setminus \R, A)\]
is defined on the class $\mathcal{C}^1(\OO_D \setminus \R,A)$ of $A$-valued functions of class $\mathcal C^1$ on $\OO_D\setminus\R$.
 In particular, when $A$ is the Clifford algebra $\Rn$, the operator $\difbar$ has the following expression
\begin{equation}\label{thetabar}
\difbar=\dd{}{x_0}+\frac{\IM(x)}{|\IM(x)|^2} \sum_{{|K|\equiv1,2\text{ (mod\,4)}}}x_{K} \, \dd{}{x_{K}}.
\end{equation}

When  the operator $\difbar$ is applied to a slice function $f$, it yields two times the slice derivative $\ddd{f}{x^c}$. Let $\dif$ be the conjugated operator of $\difbar$. Then $\dif f=2\ddd{f}{x}$ for each slice function $f$.

\begin{thm}[\cite{Gh_Pe_GlobDiff}]\label{teo1}
If $f\in \mathcal{C}^1(\OO_D)$ is a slice function, then $f$  is slice-regular if and only if $\difbar f=0$ on $\OO_D\setminus\R$.
If $\OO_D\cap\R\ne\emptyset$ and $f\in \mathcal{C}^1(\OO_D)$ (not a priori a slice function), then $f$ is  slice-regular if and only if $\difbar f=0$.
\end{thm}

As seen in the introduction, the paravector space $\R^{n+1}$ is a subspace of the quadratic cone of $\Rn$, proper if $n>2$. The action of $\difbar$ on functions defined on $\OO_D\cap\R^{n+1}$ (corresponding to the terms with $|K|=1$ in the summation \eqref{thetabar}) coincides, up to a factor 2, with the action of the \emph{radial Cauchy-Riemann operator} $\dibar_{rad}$ (cf.\ e.g.\ \cite{GuS}). An equivalent operator defined on the paravector space $\R^{n+1}$ of $\Rn$ was given in \cite{CGS}.

In the case $n=2$, the algebra $\R_2\simeq\HH$ is  four-dimensional and coincides with the quadratic cone (while paravectors form a three-dimensional subspace). When applied on functions defined on open subsets of the full algebra $\HH$, the operator $\difbar$ contains also the term with $|K|=2$ in \eqref{thetabar}.

\section{Spherical operators}

Let $f=\I(F)$ be a slice function on $\OO_D$, induced by the stem function $F=F_1+\ui F_2$, with $F_1,F_2:D\subseteq\C\to \Rn$. We recall some definitions from \cite{GhPe_AIM}:
\begin{defn}
The function $\vs f:\OO_D \to \Rn$, called \emph{spherical value} of $f$, and the function $f'_s:\OO_D \setminus \R \to \Rn$, called  \emph{spherical derivative} of $f$, are defined as
\[
\vs f(x):=\frac{1}{2}(f(x)+f(x^c))
\quad \text{and} \quad
f'_s(x):=\frac{1}{2}\IM(x)^{-1}(f(x)-f(x^c)).
\] 
\end{defn}

If $x=\alpha+\beta J\in\OO_D$, $z=\alpha+\ui\beta\in D$, then $\vs f(x)=F_1(z)$ and $f'_s(x)=\beta^{-1} F_2(z)$. Therefore $\vs f$ and $f'_s$ are slice functions, constant on every set $\SS_x=\alpha+\beta\,\SS_{\Rn}$. They are slice-regular only if $f$ is locally constant.
Moreover, the formula
  \begin{equation}\label{formula}
  f(x)=\vs f(x)+\IM(x)f'_s(x)
  \end{equation}
holds for each $x\in\OO_D\setminus \R$. If $F\in\mathcal{C}^1$, the formula holds also for $x\in\OO_D\cap\R$. In particular, if $f$ is slice-regular, $\sd{f}$ extends with the values of the slice derivative $\ddd{f}{x}$ on the real points $x\in\OO_D\cap\R$.

Since  the paravector space $\R^{n+1}$ is contained in the quadratic cone $\Q_{\Rn}$, we can consider the restriction of a slice function on domains of the form $\OO=\OO_D\cap\R^{n+1}$ in $\R^{n+1}$. Thanks to the representation formula (see e.g.~\cite[Prop.6]{GhPe_AIM}), this restriction uniquely determines the slice function. We will therefore use the same symbol to denote the restriction. 

To simplify notation, in the following we will denote the partial derivatives $\ddd{}{x_i}$ also with the symbol $\dds i$ ($i=0,\ldots, n$).

For any $i,j$ with $1\le i<j\le n$, let $L_{ij}=x_i\dds j-x_j\dds i$ be the angular momentum operators and $\Gamma=-\sum_{i<j}e_{ij}L_{ij}$ the \emph{spherical Dirac operator} on $\Rn$ (see e.g.~\cite[\S2.1.5]{GuS}, \cite[\S8.7]{BDS} or \cite{SprossigZAA}). 
The operators $L_{ij}$ are tangential differential operators for the spheres $\SS_x\cap\R^{n+1}=\alpha+\beta\,\SS^{n-1}$.
The spherical Dirac operator $\Gamma$ factorizes the Laplace-Beltrami operator $\Delta_{LB}=\sum_{i<j}L_{ij}^2$ on $\SS^{n-1}$ since $\Delta_{LB}=\Gamma(-\Gamma+n-2).$
We show that the function obtained applying the operator $\Gamma$ to a slice function $f$ is equal, up to a multiple of $\IM(x)$, to the spherical derivative $\sd f$.

\begin{prop}\label{teo2}
Let $\OO=\OO_D\cap\R^{n+1}$ be an open subset of $\R^{n+1}$.
For each slice function $f:\OO\to\Rn$ of class $\mathcal{C}^1(\OO)$, the following formulas hold on $\OO\setminus\R$:
\begin{itemize}\setlength\itemsep{0.5em}
	\item[(a)]
$\Gamma f=(n-1)\IM(x)\sd f$.
\item[(b)]
$\dibar f-\difbar f=(1-n)f'_s$.
\end{itemize}
\end{prop}
\begin{proof}
Let $f$ be a slice function.  Since the functions $\vs f$ and $\sd f$ are constant on the spheres $\SS_x$, every $L_{ij}$ vanishes on them. Using formula \eqref{formula} and Leibniz rule we get
\[
L_{ij}f=L_{ij}(\vs f(x)+\IM(x)f'_s(x))=L_{ij}(\IM(x))\sd f.
\]
A direct computation gives 
\[\Gamma x=-\sum_{i<j}e_{ij}L_{ij}(\IM(x))=-\sum_{i<j}e_{ij}(x_ie_j-x_je_i)=(n-1)\IM(x).
\]
It follows that
\[\Gamma f=-\sum_{i<j}e_{ij}L_{ij}f=-\sum_{i<j}e_{ij}L_{ij}(\IM(x))\sd f=(n-1)\IM(x)\sd f
\]
and point (a) is proved.
To prove (b), we can use the decomposition of the Cauchy-Riemann operator given in \cite{SprossigZAA}:
\[\dibar=\dds 0+\omega\ell_\omega+\frac1{|\IM(x)|}L
\]
where $\omega=\frac{\IM(x)}{|\IM(x)|}$, $\ell_\omega=\frac{1}{|\IM(x)|}\sum_{i=1}^n x_i\dds i$ and $L=\omega\Gamma$. Since $f$ depends only on paravector variables, $\difbar f$ coincides with the radial part $(\dds 0+\omega\ell_\omega)f$ of $\dibar f$. Then $\dibar f-\difbar f=\frac1{|\IM(x)|}Lf=\frac1{|\IM(x)|}\omega \Gamma f=\omega^2(n-1) \sd f=(1-n)\sd f$.
\end{proof}

\begin{cor}\label{cor1}
Let $f:\OO\subseteq\R^{n+1}\to\Rn$ be a  slice function of class $\mathcal{C}^1(\OO)$. Then
\begin{itemize}\setlength\itemsep{0.2em}
\item[(a)]
$f$ is slice-regular if and only if $\dibar f=(1-n)f'_s$.
\item[(b)]
Let  $n>1$. Then $f$ is slice-regular and monogenic (i.e.\ it belongs to the kernel of $\dibar$) if and only if $f$ is (locally) constant.
\item[(c)]
$\partial f-\dif f=(n-1)f'_s$ and
$\dif f'_s=\partial f'_s$.
\end{itemize}
\end{cor}
\begin{proof}
The first statement is immediate from point (b) of Proposition\ \ref{teo2} and Theorem\ \ref{teo1}. If $\dibar f=\difbar f=0$, then $\sd f\equiv0$. This means that the component $F_2$ of the inducing stem function $F$ vanishes identically. From the holomorphicity of $F$ it follows that $F_1$ is locally constant, and then also $f$ is locally constant. The first statement in (c) is a consequence of point (b) of Proposition\ \ref{teo2}, taking account of the equalities
\[\partial+\dibar=2\dds 0=\dif+\difbar.\]
The last statement comes from the property  $(f'_s)'_s=0$, which holds for every slice function $f$.
\end{proof}

Statement (b) of the previous Corollary shows that when $n>1$ the two function theories, the one of monogenic functions and the one of slice-regular functions, are really skew. This is in contrast with the classical case ($n=1$), when $\dibar f=\difbar f$ and the two theories coincide.

\section{The Laplacian of slice functions}\label{TheLaplacian}

Let $f=\I(F)$ be a slice function on $\OO_D$, with $F=F_1+\ui F_2$ a stem function with real analytic components $F_1$, $F_2$. Since $F(\overline z)=\overline{F(z)}$ for every $z=\alpha+\ui\beta$, the functions $F_1,F_2:D\subseteq\C\to \Rn$ are, respectively, even and odd functions with respect to the variable $\beta$. Therefore there exist $G_1$ and $G_2$, again real analytic, such that
 \[F_1(\alpha, \beta)= G_1(\alpha,\beta^2),\quad F_2(\alpha, \beta)=\beta G_2(\alpha,\beta^2).\]
If $x=\alpha+\beta J$, $z=\alpha+\ui\beta $, then 
\begin{align}
\vs f(x)&=G_1(\alpha,\beta^2)=G_1(\RE(x),|\IM(x)|^2),\label{g1}
\\
f'_s(x)&=G_2(\alpha,\beta^2)=G_2(\RE(x),|\IM(x)|^2).\label{g2}
\end{align}
The functions $G_1$ and $G_2$ are useful in the computation of the Laplacian of the spherical derivative and of the spherical value of a slice regular function. Let $\OO=\OO_D\cap\R^{n+1}$.

Observe that if $f$ is slice regular, then $F_1$ and $F_2$ have harmonic real components with respect to the two-dimensional Laplacian $\Delta_2$ of the plane. For $j=1,2$, let  $\partial_1G_j(u,v)$ stand for the partial derivative $\ddd{G_j}u(u,v)$ and $\partial_2G_j(u,v)$ for the partial derivative $\ddd{G_j}v(u,v)$.

\begin{thm}\label{laplacian}
Let $\OO=\OO_D\cap\R^{n+1}$ be an open subset of $\R^{n+1}$.
Let $f=\I(F):\OO\to\Rn$ be (the restriction of) a slice-regular function. Let $f'_s(x)=G_2(\RE(x),|\IM(x)|^2)$ as in \eqref{g2} and
let  $\Delta_{n+1}$ be the Laplacian operator on $\R^{n+1}$. Then it holds:
\begin{itemize}\setlength\itemsep{0.1em}
\item[(a)]
\[\Delta_{n+1}f'_s(x)=2(n-3)\,\partial_2 G_2(\RE(x),|\IM(x)|^2).\]
\item[(b)]
For each $k=1,2,\ldots, \left[\frac{n-1}2\right]$,
\[\Delta_{n+1}^k f'_s(x)=2^k(n-3)(n-5)\cdots(n-2k-1)\,\partial_2^k G_2(\RE(x),|\IM(x)|^2).\]
\item[(c)]
\[\Delta_{n+1}f'_s(x)=\frac{n-3}{|\IM(x)|^2}\left(\dd{\vs f}{x_0}(x)-f'_s(x)\right).\]
\end{itemize}
More precisely, the formulas in (a) and (b) hold if and only if $F_2$ has harmonic components on $D$.
\end{thm}
\begin{proof}
%The computations are similar to those made in %\cite{Qian1997} and 
%\cite[Theorem 11.33]{GHS}. 
Let $x_0=\RE(x)$, $r=|\IM(x)|$. By direct computation, from \eqref{g2} we get
\begin{equation}\label{delta2}
\Delta_2F_2(\alpha,\beta)=\left(\dds 1^2+4\beta^2\,\dds 2^2 +6\,\dds 2 \right)G_2(\alpha,\beta^2)
\end{equation}
and
\begin{align*}
\Delta_{n+1}G_2(x_0,r^2)&=\frac{\partial^2{G_2}}{\partial x_0^2}(x_0,r^2)+\sum_{i=1}^n\frac{\partial^2{G_2}}{\partial x_i^2}(x_0,r^2)\\
&=\left(\dds 1^2 +4r^2\dds 2^2 +2n\,\dds 2\right) G_2(x_0,r^2).
\end{align*}
Therefore $\Delta_2F_2=0$ on $D$ if and only if $\Delta_{n+1}\sd f(x)=\Delta_{n+1}G_2(x_0,r^2)=(2n-6)\,\dds 2 G_2(x_0,r^2)$. This proves (a). To obtain (b) we use induction on $k$, starting from the case $k=1$ given by (a). For every $k$ with $1<k\le\left[\frac{n-1}2\right]-1$, if $\Delta_2F_2(\alpha,\beta)=0$ it holds
\begin{align*}
\Delta_{n+1}\dds 2^k G_2(x_0,r^2)&=\left(\dds 1^2\dds 2^k +4r^2\dds 2^{k+2} +2n\,\dds 2^{k+1}\right) G_2(x_0,r^2)\\
&=\left(\dds 2^k\left(\dds 1^2 +4r^2\dds 2^2 +2n\,\dds 2\right)-4k\,\dds 2^{k+1}\right) G_2(x_0,r^2)\\
&=\left(\dds 2^k\left(-6\,\dds 2 +2n\,\dds 2\right)-4k\dds 2^{k+1}\right) G_2(x_0,r^2)\\
&=2(n-2k-3)\,\dds 2^{k+1}  G_2(x_0,r^2).
\end{align*}
By the induction hypothesis 
\begin{align*}
\Delta_{n+1}^{k+1}\sd f(x)&=2^k(n-3)(n-5)\cdots(n-2k-1)\,\Delta_{n+1}\partial_2^k G_2(x_0,r^2)\\&=
2^k(n-3)(n-5)\cdots(n-2k-1)\, 2(n-2k-3)\,\dds 2^{k+1} G_2(x_0,r^2)
\end{align*}
and (b) is proved.
Statement (c) follows from the holomorphicity of $F$. Since $\dds\alpha F_1(\alpha,\beta)=\dds\beta F_2(\alpha,\beta)=\dds\beta (\beta G_2(\alpha,\beta^2))=G_2(\alpha,\beta^2)+2\beta^2\dds 2 G_2(\alpha,\beta^2)$, it holds, for $r\ne0$,
\[
\dds 2 G_2(x_0,r^2)=\frac1{2r^2}\left(\dds 1 F_1(x_0,r^2)-G_2(x_0,r^2)\right)=\frac1{2r^2}\left(\dd {\vs f}{x_0} (x)-\sd f (x)\right).
\]
Together with (a), this proves (c).
\end{proof}

These results take a particularly attractive form when the paravectors space is four-dimensional, i.e.\ $n=3$.

\begin{cor}\label{cor2}
Let $f:\OO\subseteq\R^4\to\R_3$ be (the restriction of) a slice-regular function. Then it holds:
\begin{itemize}\setlength\itemsep{0.1em}
\item[(a)]
  The spherical derivative $f'_s$ is harmonic on $\OO$, i.e.\ its eight real components are harmonic functions.
\item[(b)]
 The following generalization of Fueter-Sce Theorem for the Clifford algebra $\R_3$ holds:
\[\dibar\Delta_4f=\Delta_4 \dibar f=-2\Delta_4f'_s=0.\]
\item[(c)]
 $\Delta_4^2 f=0$, i.e.\ every slice regular function on $\R_3$ is biharmonic on $\OO$.
\item[(d)]
$\Delta_4 f=-4\,\dd{\sd f}{x}$. Therefore also $\dd{\sd f}{x}$ is harmonic on $\OO$.
\end{itemize}
\end{cor}
\begin{proof}
The first statement is immediate from point (a) of Theorem~\ref{laplacian}. Point (b) is a consequence of (a) and of Corollary~\ref{cor1}. Statement (c) follows from (b) and the factorization $\partial\dibar=\Delta_4$.
Since $2\ddd{f}{x}=\dif f$ (cf.\ \cite{Gh_Pe_GlobDiff}) for any slice function, from Corollary~\ref{cor1} it follows that
\[
4\dd{\sd f}{x}=2\dif(\sd f)=\partial (2\sd f)=-\partial \dibar f=-\Delta_4f.
\]
This proves (d).
\end{proof}

\begin{rem}
If $f$ is only \emph{slice-harmonic} on $\OO_D$, i.e.\ $f=\I(F)$ is induced by a harmonic stem function $F$ on $D\subseteq\C$, then $F_2$ has harmonic real components and therefore $\sd f$ is still harmonic.
\end{rem}

\begin{examples}
(1)\quad 
If $f=x^3$, a slice-regular function, then \[\dibar f=-2\sd f=-2 \left(3 x_0^2-x_1^2-x_2^2-x_3^2\right)\] is harmonic on $\R^4$ and $\Delta_4 f=-4\left(3x_0+x_1e_1+x_2e_2+x_3e_3\right)$ is monogenic.

(2)\quad
If $f= (x^c)^3$, a slice-harmonic function, then \[\sd f=-3 x_0^2+x_1^2+x_2^2+x_3^2\] is harmonic on $\R^4$ while $\Delta_4 f=4\left(-3x_0+x_1e_1+x_2e_2+x_3e_3\right)$ is not monogenic.

(3)\quad
Let $f=x\left(1-\frac{\IM(x)}{|\IM(x)|}e_1\right)$. The function $f$ is slice regular on $\Q_{\R_3}\setminus\R$, a set that contains $\R^4\setminus\R$. Then $\sd f=1-\frac{x_0}{|\IM(x)|}e_1$ is harmonic on $\R^4\setminus\R$. The Laplacian
\[\Delta_4f=\frac2{|\IM(x)|^3}\left(-x_0x_1+(x_1^2+x_2^2+x_3^2)e_1-x_0x_2e_{12}-x_0x_3e_{13}\right)
\]
is monogenic on $\R^4\setminus\R$.

\end{examples}

Now consider the higher dimensional case, with $n>3$ odd.

\begin{cor}\label{cor:n_odd}
Let $n>3$ odd. If $f:\OO\subseteq\R^{n+1}\to\R_n$ is (the restriction of) a slice-regular function, then
\begin{itemize}\setlength\itemsep{0.1em}
\item[(a)]
  $(\Delta_{n+1})^{\frac{n-3}2}f'_s$ is harmonic on $\OO$.
\item[(b)]
 The following generalization of Fueter-Sce Theorem for $\Rn$ holds:
\[ \dibar(\Delta_{n+1})^{\frac{n-1}2}f=(\Delta_{n+1})^{\frac{n-1}2} \dibar f=(1-n)\,(\Delta_{n+1})^{\frac{n-1}2} f'_s=0.\]
\item[(c)]
 $(\Delta_{n+1})^{\frac{n+1}2} f=0$, i.e.\ every slice regular function on $\Rn$ is polyharmonic.
\end{itemize}
\end{cor}
\begin{proof}
The proof follows the same lines as above, using point (b) of Theorem~\ref{laplacian} instead of (a).
\end{proof}

The harmonicity properties of slice-regular functions imply a stronger form of the Liouville's Theorem for entire slice regular functions. See \cite{CSS} for the generalization of the classical result to slice monogenic functions. In the next Corollary we give also a new proof of this last result (at least in the case of $n$ odd).

\begin{cor}\label{LiouvilleClifford}
Let $n\ge3$ be odd. Let $f\in\mathcal {SR}(\Q_{\R_n})$ be an entire slice regular function. 
If $f$ is bounded on $\R^{n+1}$, then $f$ is constant. 
If the spherical derivative $f'_s$ is bounded (equivalently, if $\dibar f$ is bounded) on $\R^{n+1}$, then $f$ is a left-affine function, of the form $f(x)=a+xb$ ($a,b\in\Rn$). 
\end{cor}
\begin{proof}
Let $f=\I(F)$ be induced by the holomorphic stem function $F=F_1+\ui F_2$, with $F_1,F_2:\C\to\Rn$. 
If $f$ is bounded, then the real components of $f$ are polyharmonic and bounded on $\R^{n+1}$. Then $f$ is constant from the Liouville's Theorem for polyharmonic functions. 

If $\sd f$ is bounded, the real components of its continuous extension to $\R^{n+1}$ are polyharmonic and bounded and then constant. Therefore $F_2(\alpha,\beta)=\beta b$, with $b\in\Rn$. By the Cauchy-Riemann equations, it follows that $F_1(\alpha,\beta)=a+\alpha b$, with $a\in\Rn$. Therefore $f(x)=(a+x_0b)+\IM(x)b=a+xb$.
\end{proof}

As regards the spherical value of a slice-regular function, we can still compute its Laplacian. In general, even in the four-dimensional case, the spherical value is not a harmonic function, nonetheless in the even-dimension\-al case it is always polyharmonic.

\begin{thm}\label{vs}
Let $f=\I(F):\OO\subseteq\R^{n+1}\to\Rn$ be (the restriction of) a slice-regular function. Let $\vs f(x)=G_1(\RE(x),|\IM(x)|^2)$ as in \eqref{g1}.
It holds:
\begin{itemize}\setlength\itemsep{0.1em}
\item[(a)]
\[\Delta_{n+1}\vs f(x)=\vs{(\Delta_{n+1} f)}(x)=2(n-1)\,\partial_2 G_1(\RE(x),|\IM(x)|^2).\]
\item[(b)]
For each $k=1,2,\ldots, \left[\frac{n-1}2\right]$
\[\Delta_{n+1}^k \vs f(x)=2^k(n-1)(n-3)\cdots(n-2k+1)\,\partial_2^k G_1(\RE(x),|\IM(x)|^2).\]
\item[(c)]
\[\Delta_{n+1}\vs f(x)=(1-n)\dd{f'_s}{x_0}(x).\]
\item[(d)]
When $n=3$, $\Delta_4^2\vs f=0$. In general, if $n$ is odd, $(\Delta_{n+1})^{\frac{n+1}2} \vs f=0$.
\end{itemize}
\end{thm}
\begin{proof}
Let $x_0=\RE(x)$, $r=|\IM(x)|$. By direct computation, from \eqref{g1} we get
\begin{equation}\label{delta1}
\Delta_2F_1(\alpha,\beta)=\left(\dds 1^2+4\beta^2\,\dds 2^2 +2\,\dds 2 \right)G_1(\alpha,\beta^2).
\end{equation}
The proofs of (a) and (b) follows the same lines as the corresponding proofs of Theorem\ \ref{laplacian}, using \eqref{delta1} in place of \eqref{delta2}.
We prove (c): since $\dds\alpha F_2(\alpha,\beta)=-\dds\beta F_1(\alpha,\beta)=-2\beta\dds 2 G_1(\alpha,\beta^2)$, it holds, for $r\ne0$,
\[
\dds 2 G_1(x_0,r^2)=-\frac1{2r}r\,\dds 1 G_2(x_0,r^2)=-\frac1{2}\dd {\sd f}{x_0}(x).
\]
Together with (a), this proves (c). Finally, (d) is immediate from (b).
\end{proof}

Again, points (a) and (b) remain valid if $f$ is slice-harmonic.

\section{The four-dimensional case: zonal harmonics and the Poisson kernel}\label{4dcase}

Thanks to Corollary~\ref{cor2}, for any polynomial $f(x)=\sum_{m=0}^d x^ma_m$ with coefficients in $\R_3$, the spherical derivative
\[ f'_s(x)=\sum_{m=0}^d \sd {(x^m)} a_m\]
 is a harmonic polynomial on $\R^4$.
In particular, for every $m\in\N$ the spherical derivative $\sd{(x^m)}=-\frac12\dibar(x^m)$ of a Clifford power $x^m$ is a homogeneous harmonic polynomial of degree $m-1$ in the variables $x_0,x_1,x_2,x_3$, with real coefficients.  Observe that $\sd{(x^m)}$ can be written as
\begin{align*}
\sd{(x^m)}&=\frac{\IM(x)^{-1}}2\left(x^{m}-(x^c)^{m}\right)=(x-x^c)^{-1}\left(x^{m}-(x^c)^{m}\right)\\
&=\sum_{k=0}^{m-1}x^{m-k-1}(x^c)^k=\sum_{\nu=0}^{\left[\frac{m-2}2\right]}t(x^{m-1-2\nu})n(x)^\nu + n(x)^{\frac{m-1}2}
\end{align*}
(where the last term is present only if $m$ is odd).

Let $\mathbb B$ be the open unit ball in $\R^4$.  Let $\mathcal Z_{m}(x,a)$ denote the four-dimen\-sion\-al \emph{(solid) zonal harmonic} of degree $m$ with pole $a\in\partial\mathbb B$ (see e.g.~\cite[Ch.5]{HFT}).
From the uniqueness properties of zonal harmonics and their invariance with respect to four-dimensional rotations, we get the following result. 

\begin{prop}\label{zonal}
The spherical derivatives of Clifford powers $x^m$ coincide on $\R^4$, up to a constant, with the zonal harmonics with pole $1\in\partial\mathbb B$.
More precisely, for every $m\ge1$ and every $a\in\partial\mathbb B$, it holds: 
\begin{itemize}\setlength\itemsep{0.1em}
\item[(a)]
$\mathcal Z_{m-1}(x,1)=m \sd {(x^m)}$. Therefore $\dibar(x^m)=-\frac2m\mathcal Z_{m-1}(x,1)$.
\item[(b)]
$\mathcal Z_{m-1}(x,a)=\mathcal Z_{m-1}(xa^c,1)=m {\sd {(x^m)}}_{|x=xa^c}$\\$=m\sum_{k=0}^{m-1}(xa^c)^{m-k-1}(ax^c)^k$.
\item[(c)]
$\sd {(x^{-m})}=-\mathcal K\left[\sd {(x^m)}\right]$, where $\mathcal K$ is the Kelvin transform in $\R^4$. The functions $\sd {(x^{-m})}$ are harmonic on $\R^4\setminus\{0\}$.
\item[(d)]
The restriction of $\sd{(x^m)}$ to the unit sphere $\partial\mathbb B$ is equal to the Gegenbauer polynomial $C^{(1)}_{m-1}(\RE(x))$.
\end{itemize}
\end{prop}
\begin{proof}
Let $\alpha,\beta\in\R$ with $\beta>0$ and $\alpha^2+\beta^2=1$. The spherical derivatives $\sd {(x^m)}$ are constant on every ``parallel" $\SS_x=\alpha+\beta\SS^{2}$ in $\partial\mathbb B$ orthogonal to the real axis. From \cite[Theorem~5.37]{HFT} it follows that $\sd {(x^m)}$ is a constant multiple of $\mathcal Z_{m-1}(x,1)$. To determine the constant, it is sufficient to compute $\sd {(x^m)}$ at $x=1$. On real points the spherical derivative coincides with the slice derivative $\ddd f x$, and then ${\sd {(x^m)}}_{|x=1}={\ddd {x^m}x}_{|x=1}=m{x^{m-1}}_{|x=1}=m$. Since $\mathcal Z_{m-1}(1,1)=m^2$, (a) is proved. Point (b) is a consequence of the rotational properties of zonal harmonics (cf.~\cite[5.27]{HFT}). Statement (c) follows from direct computation:
\[
\sd {(x^{-m})}=%\frac{\IM(x)^{-1}}2\left(x^{-m}-(x^c)^{-m}\right)=
(x-x^c)^{-1}\left(\frac{(x^c)^m}{|x|^{2m}}-\frac{x^m}{|x|^{2m}}\right)
\]
and then
\begin{align*}
\mathcal K\left[\sd {(x^{-m})}\right]&=
|x|^{-2}\left(\frac{x}{|x|^2}-\frac{x^c}{|x|^2}\right)^{-1}\left(\left(\frac{x^c}{|x|^2}\right)^m-\left(\frac{x}{|x|^2}\right)^m\right) \left|\frac{x^c}{|x|^2}\right|^{-2m}\\
&=(x-x^c)^{-1}((x^c)^m-x^m)=
-\sd {(x^{m})}.
\end{align*}
Since $\mathcal K[\mathcal K[f]]=f$, this proves (c). Statement (d) follows from (a) and a well-known property of zonal harmonics. Note that $C^{(1)}_{m-1}(1)=m$ for each $m\ge1$.
\end{proof}

\begin{cor}
The Clifford powers $x^m$ of the paravector variable of $\R_3$ can be expressed in terms of the four-dimensional zonal harmonics: for each $m\ge1$,
\begin{align*}
x^m&=\frac1{m+1}\mathcal Z_m(x,1)+\frac{(x-2\RE(x))}m \mathcal Z_{m-1}(x,1)\\
	&=\frac1{m+1}\mathcal Z_m(x,1)-x^c\frac{1}m \mathcal Z_{m-1}(x,1).
\end{align*}
\end{cor}
\begin{proof}
Let $x_0=\RE(x)$. Applying the Leibniz rule for the spherical derivative (cf.~\cite[\S5]{GhPe_AIM}), we get
\[
\sd{(x^{m+1})}=\sd{(x\cdot x^{m})}=\sd{(x)}\vs{(x^{m})}+\vs{(x)}\sd{(x^{m})}=\vs{(x^{m})}+x_0\sd{(x^{m})}.
\]
Therefore
\[
\vs{(x^{m})}=\sd{(x^{m+1})}-x_0\sd{(x^{m})}=\frac1{m+1}\mathcal Z_m(x,1)-x_0 \frac1{m}\mathcal Z_{m-1}(x,1)
\]
and
\[
x^{m}=\frac1{m+1}\mathcal Z_m(x,1)-x_0 \frac1{m}\mathcal Z_{m-1}(x,1)+\IM(x)\frac1{m}\mathcal Z_{m-1}(x,1).
\]
\end{proof}

Let $\mathcal P(x,a)=\frac{1-|x|^2}{|x-a|^4}$ be the Poisson kernel for the unit ball $\mathbb B$ in $\R^4$ ($x\in\mathbb B$, $a\in \partial\mathbb B$). This harmonic kernel is related with the slice-regular function induced by a famous holomorphic function.
Let $f_K(x)=(1-x)^{-2}x$ be the \emph{Cliffordian Koebe function}. It is the  slice-preserving slice-regular function induced by the classical Koebe function $F_K(z)=(1-z)^{-2}z$. 

\begin{cor}
The Cliffordian Koebe function $f_K(x)=(1-x)^{-2}x$ is slice-regular on $\Q_{\R_3}\setminus\{1\}\supset \mathbb B$ and has the following properties. For every $x\in\mathbb B$,
\[\sd {(f_K)}(x)=\mathcal P(x,1)=\frac{1-|x|^2}{|x-1|^4}.\]
 For every $a\in\partial\mathbb B$ and $x\in\mathbb B$,
\[\sd {(f_K)}(xa^c)=\mathcal P(x,a)=\frac{1-|x|^2}{|x-a|^4}.\]
\end{cor}
\begin{proof}
The formulas can be checked directly or by means of the relation of $\sd{(x^m)}$ with zonal harmonics proved in Proposition~\ref{zonal}.
The power series $\sum_{m=0}^{\infty}(m+1)z^{m+1}$ converges to $F_K(z)$ on the complex unit disc. This implies the expansion  $f_K(x)=\sum_{m=0}^{\infty}(m+1)x^{m+1}$ on $\mathbb B$. Therefore, for every $x\in\mathbb B$, it holds
\[
\sd{(f_K)}(x)=\sum_{m=0}^{\infty}(m+1)\sd{(x^{m+1})}=\sum_{m=0}^{\infty}\mathcal Z_m(x,1)=\mathcal P(x,1),
\]
where the last equality follows from the zonal harmonic expansion  $\mathcal P(x,a)=\sum_{m=0}^\infty\mathcal Z_m(x,a)$ (cf.~\cite[Theorem~5.33]{HFT}). The last statement is a consequence of point (b) of Proposition~\ref{zonal}.
\end{proof}

\section{The quaternionic case}\label{Thequaternioniccase}

When $n=2$, the Clifford algebra $\R_{2}$ is isomorphic to the field $\HH$ of quaternions. In this case the paravector space has dimension three and then Corollary~\ref{cor:n_odd} and its consequences are not applicable. However, similar results still hold since the computations made in Proposition~\ref{teo2} on the paravector space can be repeated anytime there is a real subspace of the quadratic cone containing the real axis.
The simplest example of this setting is given by the quaternions, where the quadratic cone coincides with the whole algebra: $\Q_\HH=\HH$.

By means of the identifications $e_1=i$, $e_2=j$, $e_{12}=ij=k$, in the coordinates $(x_0,x_1,x_2,x_3)$ of $x=x_0+x_1i+x_2j+x_3k\in\HH$, the differential operator $\difbar$ takes the form \cite{Gh_Pe_GlobDiff}
\[
\difbar=\dd{}{x_0}+\frac{\IM(x)}{|\IM(x)|^2} \sum_{i=1}^3x_{i} \, \dd{}{x_{i}}.
\]
For every slice function $f:\OO\subseteq\R^{3}\to\R_{2}$, of class $\mathcal{C}^1$ on a domain $\OO$ in the three-dimensional space of (quaternionic) paravectors, Proposition~\ref{teo2} gives
\begin{equation}\label{formula1}
\dibar f-\difbar f=-f'_s.
\end{equation}
If we consider the whole quaternion algebra, we must instead use the Cauchy-Riemann-Fueter operator
\[\dcrf =\dd{}{x_0}+i\dd{}{x_1}+j\dd{}{x_2}+k\dd{}{x_3}=\dibar+k\dd{}{x_3}.\]
Let $\dif$ and $\partial_{\scriptscriptstyle CRF}$ be the conjugated differential operators:
\[
\dif=\dd{}{x_0}-\frac{\IM(x)}{|\IM(x)|^2} \sum_{i=1}^3x_{i} \, \dd{}{x_{i}}\quad
\text{and}
\quad\partial_{\scriptscriptstyle CRF} =\dd{}{x_0}-i\dd{}{x_1}-j\dd{}{x_2}-k\dd{}{x_3}.\]
The Cauchy-Riemann-Fueter operator $\dcrf$ factorizes the Laplacian operator of $\R^4$: \[\partial_{\scriptscriptstyle CRF}\dcrf=\dcrf\partial_{\scriptscriptstyle CRF}=\Delta_4.\]
For any $i,j$ with $1\le i<j\le 3$, let $L_{ij}=x_i\dds j-x_j\dds i$ and let
\[\Gamma=%-\sum_{i<j}e_ie_jL_{ij}=
-iL_{23}+jL_{13}-kL_{12}\] 
be the \emph{quaternionic spherical Dirac operator} on $\IM(\HH)$. The operators $L_{ij}$ are tangential differential operators for the spheres $\SS_x=\alpha+\beta\,\SS^{2}$.
For the Cauchy-Riemann-Fueter operator the analogous of Proposition~\ref{teo2} is the following result. Compare point (b) with formula \eqref{formula1}.

\begin{prop}\label{propH}
Let $\OO=\OO_D$ be an open circular domain in $\HH$. For every slice function $f:\OO\to\HH$, of class $\mathcal{C}^1(\OO)$, the following formulas hold on $\OO\setminus\R$:
\begin{itemize}\setlength\itemsep{0.5em}
\item[(a)]
$\Gamma f=2\IM(x)\sd{f}$.
\item[(b)]
$\dcrf f-\difbar f=-2f'_s$.
\end{itemize}
\end{prop}
\begin{proof}
The proof of (a) is the same as the one given for point (a) of Proposition~\ref{teo2}. To prove (b), set $r=|\IM(x)|$, $\omega=\frac{\IM(x)}{|\IM(x)|}$, $\ell_\omega=\frac{1}{|\IM(x)|}\sum_{i=1}^3 x_i\dds i$ and $L=\omega\Gamma$. It holds
\[
r\omega\left(i\dds i+j\dds 2+k\dds 3\right)=\IM(x)\left(i\dds i+j\dds 2+k\dds 3\right)=-\sum_{i=1}^3x_i\dds i-\Gamma.
\]
Therefore 
\[\dcrf =\dds 0+\left(i\dds i+j\dds 2+k\dds 3\right)=\dds 0+\frac{\omega}r\left(\sum_{i=1}^3x_i\dds i+\Gamma\right)=\dds 0+\omega\ell_\omega+\frac1 rL.
\]
Since $\difbar f$ coincides with the radial part $(\dds 0+\omega\ell_\omega)f$ of $\dcrf f$, we get $\dcrf f-\difbar f=\frac1{|\IM(x)|}Lf=\frac1{|\IM(x)|}\omega \Gamma f=2\omega^2 \sd f=-2\sd f$ and (b) is proved. 
\end{proof}

\begin{cor}\label{corH}
Let $\OO=\OO_D$ be an open circular domain in $\HH$. Let  $f:\OO\to\HH$ be a slice function of class $\mathcal{C}^1(\OO)$. Then
\begin{itemize}\setlength\itemsep{0.1em}
\item[(a)]
$f$ is slice-regular if and only if $\dcrf f=-2f'_s$.
\item[(b)]
$f$ is slice-regular and Fueter-regular (i.e.\ it belongs to the kernel of $\dcrf$) if and only if $f$ is (locally) constant.
\item[(c)]
$\partial_{\scriptscriptstyle CRF}f-\dif f=2f'_s$ and
$\dif f'_s=\partial_{\scriptscriptstyle CRF}f'_s$.
\end{itemize}
\end{cor}
\begin{proof}
The proofs of (a) and (b) are the same as the ones given in Corollary~\ref{cor1}. From $\dcrf+\partial_{\scriptscriptstyle CRF}=2\dds 0=\dif+\difbar$ and point (b) of Proposition~\ref{propH}, it follows the first statement in (c). The last statement comes from the property  $(f'_s)'_s=0$, which holds for every slice function $f$.
\end{proof}

\begin{thm}\label{teo12}
Let $\OO=\OO_D$ be an open circular domain in $\HH$. 
If $f:\OO\to\HH$ is slice-regular, then it holds:
\begin{itemize}\setlength\itemsep{0.1em}
\item[(a)]
  The spherical derivative $f'_s$ is harmonic on $\OO$ (i.e.\ its four real components are harmonic).
\item[(b)]
The following generalization of Fueter's Theorem holds:
\[\dcrf\Delta_4f=\Delta_4 \dcrf f=-2\Delta_4f'_s=0.\] 
As a consequence, $\Delta_4^2 f=0$: every quaternionic slice-regular function is biharmonic.
\item[(c)]
$\Delta_4 f=-4\,\dd{\sd f}{x}$. In particular, $\dd{\sd f}{x}$ is harmonic on $\OO$.
\end{itemize}
\end{thm}
\begin{proof}
We proceed as in Section~\ref{TheLaplacian}. Let $f=\I(F)$ and let $G_2$ be the real analytic function on $D$ such that
 $F_2(\alpha, \beta)=\beta G_2(\alpha,\beta^2)$. If $x=\alpha+\beta J$, $z=\alpha+\ui\beta $, then 
\[
f'_s(x)=G_2(\alpha,\beta^2)=G_2(\RE(x),|\IM(x)|^2).
\]
Since
\begin{equation}
\Delta_2F_2(\alpha,\beta)=\left(\dds 1^2+4\beta^2\,\dds 2^2 +6\,\dds 2 \right)G_2(\alpha,\beta^2)
\end{equation}
and
\begin{align*}
\Delta_{4}G_2(x_0,r^2)&=\frac{\partial^2{G_2}}{\partial x_0^2}(x_0,r^2)+\sum_{i=1}^3\frac{\partial^2{G_2}}{\partial x_i^2}(x_0,r^2)\\
&=\left(\dds 1^2 +4r^2\dds 2^2 +6\,\dds 2\right) G_2(x_0,r^2),
\end{align*}
$\Delta_2F_2=0$ on $D$ if and only if $\Delta_{4}\sd f(x)=\Delta_{4}G_2(x_0,r^2)=0$ on $\OO$. This proves (a). Point (b) is immediate from (a) and Proposition~\ref{propH}.
The last statement can be proved as point (d) of Corollary~\ref{cor2}.
\end{proof}

\begin{rem}
% Since $2\ddd{f}{x}=\dif f$ (cf.\ \cite{Gh_Pe_GlobDiff}), from Theorem~\ref{teo12} it follows that if $f$ is slice-regular, then
% $2\dcrf\left(\ddd{f'_s}{x}\right)=\dcrf\left(\dif f'_s\right)=\Delta_4\sd f=0$ and therefore also the slice derivative $\ddd{f'_s}{x}$ is harmonic. 
The spherical value $\vs f$ of a slice-regular function is biharmonic. This can be proved as in Theorem~\ref{vs}.
\end{rem}

\begin{rem}
As in the case of $\R_3$, if $f$ is \emph{slice-harmonic} on $\OO=\OO_D$, i.e.\ $f=\I(F)$ is induced by a harmonic stem function $F$ on $D\subseteq\C$, then $F_2$ has harmonic real components and therefore $\sd f$ is still harmonic.
\end{rem}

\begin{rem}
As it was proved in \cite{Altavilla}, the harmonic functions $\sd f(y)$ and $\ddd{\sd f}{x}(y)$ are, respectively, the first and the second coefficients of the \emph{spherical expansion} at $y$ of a slice-regular function $f$ (see \cite{Stoppato,PowerSeries}).
\end{rem}

The link existing between Clifford powers and zonal harmonics (Proposition~\ref{zonal} and its corollaries) has a quaternionic counterpart: the spherical derivatives of the quaternionic powers $x^m$ coincide on $\R^4$, up to a constant, with the four-dimensional zonal harmonics with pole $1\in\partial\mathbb B$. We do not repeat the proofs given in Section~5.

\begin{cor}
For every $m\ge1$ and every $a\in\partial\mathbb B$, it holds: 
\begin{itemize}\setlength\itemsep{0.2em}
\item[(a)]
$\mathcal Z_{m-1}(x,1)=m \sd {(x^m)}$. Therefore $\dcrf(x^m)=-\frac2m\mathcal Z_{m-1}(x,1)$.
\item[(b)]
$\mathcal Z_{m-1}(x,a)=\mathcal Z_{m-1}(x\overline a,1)=m {\sd {(x^m)}}_{|x=x\overline a}=m\sum_{k=0}^{m-1}(x\overline a)^{m-k-1}(a\overline x)^k$.
\item[(c)]
$\sd {(x^{-m})}=-\mathcal K\left[\sd {(x^m)}\right]$, where $\mathcal K$ is the Kelvin transform in $\R^4$. The functions $\sd {(x^{-m})}$ are harmonic on $\R^4\setminus\{0\}$.
\item[(d)]
$
x^m=\frac1{m+1}\mathcal Z_m(x,1)-\overline x \frac1m \mathcal Z_{m-1}(x,1).
$
\item[(e)]
The restriction of $\sd{(x^m)}$ to the unit sphere $\partial\mathbb B$ is equal to the Gegenbauer polynomial $C^{(1)}_{m-1}(\RE(x))$.
\item[(f)]
The \emph{quaternionic Koebe function} $f_K(x)=(1-x)^{-2}x$  is slice-regular on $\HH\setminus\{1\}$ and it holds
\[\sd{(f_K)}(x)=\mathcal P(x,1),\quad \sd{(f_K)}(x\overline a)=\mathcal P(x,a).\]
for every $x\in\mathbb B$, $a\in\partial\mathbb B$, where $\mathcal P(x,a)$ is the Poisson kernel of $\mathbb B$.
\end{itemize}
\qed
\end{cor}

The harmonicity properties of slice-regular functions imply also in the quaternionic case a stronger form of the Liouville's Theorem for entire slice regular functions. See \cite{GS,GSS} for the generalization of the classical result to quaternionic functions. In the next Corollary we give also a new proof of this last result.

\begin{cor}
Let $f\in\mathcal {SR}(\HH)$ be an entire slice regular function. If $f$ is bounded, then $f$ is constant. 
If the spherical derivative $f'_s$ is bounded (equivalently, if $\dcrf f$ is bounded), then $f$ is a quaternionic left-affine function, of the form $f(x)=a+xb$ ($a,b\in\HH$). 
\end{cor}
\begin{proof}
We can repeat the same arguments of the proof of Corollary~\ref{LiouvilleClifford}, using the harmonicity of $\sd f$ and the biharmonicity of $f$.
%
% Let $f=\I(F)$ with $F=F_1+\ui F_2$, with $F_1,F_2:\C\to\HH$ a holomorphic stem function. If $\sd f$ is bounded, the four real components of its continuous extension to $\HH$ are harmonic and bounded and then constant. Therefore $F_2(\alpha,\beta)=\beta b$, with $b\in\HH$. By the Cauchy-Riemann equations, it follows that $F_1(\alpha,\beta)=a+\alpha b$, with $a\in\HH$. Therefore $f(x)=(a+x_0b)+\IM(x)b=a+xb$.
%
% If $f$ is bounded, then the real components of the spherical value $\vs f$ are biharmonic and bounded on $\R^4$. Then $\vs f$ is constant from the Liouville's Theorem for biharmonic functions.
% This means that $F_1$ is constant and therefore also $F_2$ is constant. But $F_2(\alpha,0)=0$ and then $f\equiv\vs f$ is constant. 
\end{proof}

\end{document}